\documentclass[11pt,a4paper,twoside]{article}
\usepackage{amsfonts,amsmath,amssymb}
\usepackage{theorem}
\usepackage{cite}

\newif\ifdraftcites \draftcitestrue
\ifdraftcites\else
  \let\LaTeXbibitem\bibitem
  \renewcommand{\bibitem}[2][]{\LaTeXbibitem{#2}}
\fi
\usepackage{indentfirst}
\usepackage{graphicx}
\usepackage{float}
\graphicspath{{./}{verify/fig/}}

\usepackage[T1,T2A]{fontenc}
\usepackage[utf8]{inputenc}
\usepackage[russian,english]{babel}

\AtBeginDocument{} \AtBeginDocument{}

\theorembodyfont{\sl}
\newtheorem{theorem}{Theorem}[section]
\newtheorem{proposition}[theorem]{Proposition}

\newtheorem{lemma}[theorem]{Lemma}
\newtheorem{corollary}[theorem]{Corollary}
\newtheorem{remark}[theorem]{Remark}
\newtheorem{example}[theorem]{Example}

\numberwithin{equation}{section} \numberwithin{theorem}{section}
\newcommand{\card}{\operatorname{card}}
\newcommand{\bxi}{\boldsymbol{\xi}}

\makeatletter
\renewenvironment{thebibliography}[1]
 {\section*{\centerline{\rm\bfseries References}}%
 \@mkboth{\MakeUppercase\refname}{\MakeUppercase\refname}%
 \list{\@biblabel{\@arabic\c@enumiv}}%
 {\settowidth\labelwidth{\@biblabel{#1}}%
 \leftmargin\labelwidth
 \advance\leftmargin\labelsep
 \@openbib@code
  \usecounter{enumiv}%
  \let\p@enumiv\@empty
  \renewcommand\theenumiv{\@arabic\c@enumiv}}%
  \sloppy
  \clubpenalty4000
  \@clubpenalty \clubpenalty
  \widowpenalty4000%
  \sfcode`\.\@m
  \setlength{\itemsep}{-0.1cm}}
  {\def\@noitemerr
  {\@latex@warning{Empty 'thebibliography' environment}}%
 \endlist}
\renewcommand{\@biblabel}[1]{#1.}
\makeatother \hfuzz=0.5pt \tolerance=500

\begin{document}
\begin{center}
	\textbf{Numerical differentiation of functions on the half-line\\ in a weighted uniform metric}
\end{center}
\vspace*{10mm} \centerline{\textsc {S.\,G.\,Solodky $\!\!{}^{\dag}$, Y.\,A.\,Volynets $\!\!{}^{\ddag}$}}

\vspace*{5mm}
\centerline{$\!\!{}^{\dag}\!\!$ Institute of Mathematics, National Academy of Sciences of Ukraine, Kyiv}
\centerline{$\!\!{}^{\ddag}\!\!$ National University of Kyiv-Mohyla Academy, Kyiv, Ukraine}

\vspace*{8mm}

\begin{abstract}
\begin{center}
\parbox{\dimexpr\textwidth-24mm}{\small
We consider the problem of numerical differentiation of functions defined on the half-line:
the derivative $f^{(r)}$, $ r=1,2,\ldots$, is recovered from a finite set of perturbed
Fourier--Laguerre coefficients of a function $f$, whose error is measured in $\ell_p$, while
the accuracy of approximation is measured in the uniform metric with the weight $t^{\alpha}e^{-t}$. It is established
that the Wiener class $W^{\mu}_{s}$ admits such a setting for $\mu>r+1-1/s$.
A class of methods $\mathcal{S}^{\theta}$ is constructed that realize the optimal order
of accuracy of numerical differentiation on $W^{\mu}_{s}$ and use the smallest
amount of input information in order.
}
\end{center}
\end{abstract}

\vspace*{3mm}

\noindent{\textit{Key words:}
Numerical differentiation, regularization, Laguerre polynomials,
Wiener class, information complexity, optimal recovery}
\vspace*{3mm}

\noindent\textit{2020 Mathematics Subject Classification:} Primary: 65D25; Secondary: 41A25, 42C10, 65D15, 65Y20.

\section{Introduction. Preliminaries}\label{prelim}

In the present work we study the problem of optimal recovery of $f^{(r)}$ on the half-line
from perturbed Fourier--Laguerre coefficients of a function $f$ from a Wiener class.
The results of our research have as their field of applied use
such problems as, for example, differentiation of a noisy signal by an orthogonal filter \cite{LiGPe},
Cauchy problems on the half-line \cite{Shar}, \cite{KasHag},
inversion of the Laplace transform and relaxation spectra \cite{Coh},
spectral methods on unbounded domains \cite{Shen}, \cite{NeHa}.
Despite the rich history of the study of approximation properties of Laguerre polynomials,
which takes its beginning from \cite{Muck70_2}, the problem of optimization of differentiation methods
by means of Laguerre polynomials is still little studied.
Among the works devoted to this subject one can name only \cite{MelMic}.
However, the error there was estimated in the Hilbert metric.
As for the optimization of differentiation methods on the half-line in the weighted
uniform metric, the authors are not aware of a single work in this area.
It is exactly to filling this gap in the research that the present work is devoted.
Let us add that earlier in a number of papers
\cite{Sem_Sol_2024}, \cite{Sem_Sol_2025}, \cite{Kys}, \cite{Sem_Sol_2026}
the problem of optimization of differentiation methods for functions defined on a finite interval was studied.
Separately, we would also like to note the recent work of the authors \cite{SolVol}, where
the problem of optimization of summation methods for functions defined on the half-line was studied.
Thus, the present paper can be regarded as a natural
continuation of \cite{SolVol}.

Let us draw attention to one more fact. The error of recovery of the derivative is sought here
in a fixed space that does not depend on the order of differentiation.
Thereby our model differs essentially from the overwhelming majority of the works
of predecessors, where the weight was matched to the order of the derivative, which made it possible
to reduce the error.
The mismatched weight in our work is neither a simplification nor an oversight.
Our approach is motivated by practical considerations: if the output norm does not depend on $r$,
then one method can optimally approximate derivatives of different orders in the same metric
(see remark \ref{rem_onemeth}).
Owing to this, optimal accuracy is achieved in solving those problems
in whose setting derivatives of different orders of the differentiated function occur.
It is exactly such an approach that provides high efficiency in solving problems of this kind
in practice.

For the presentation of the material we shall need the following notation and notions.
Throughout the paper $\alpha\ge 0$ is a real number (not necessarily an integer), $t>0$, and
$r\ge 1$ is a fixed integer, the order of the derivative to be recovered.
On the half-line $(0,\infty)$ we use the weight $w(t)=t^\alpha e^{-t}$; this weight does \emph{not} depend
on~$r$.\looseness=-1

By $L_{2,w}=L_{2,w}(0, \infty)$ we denote the weighted Hilbert space of
real-valued functions $f$ that are square-summable on $(0,\infty)$ with the weight $w$, with the
inner product and norm
$$
\langle f, g\rangle=\int_{0}^{\infty} t^{\alpha} e^{-t} f(t) g(t)\, d t ,
\qquad
\|f\|_{2}^2 := \int_{0}^\infty t^{\alpha} e^{-t} |f(t)|^2 \, d t .
$$
The Laguerre polynomials (see, for example, \cite[ch.\,5]{Szego}) can be represented in the form
$$
L_k^{(\alpha)}(t)
= \sum_{m=0}^k \frac{\Gamma(k+\alpha+1)}{\Gamma(m+\alpha+1) (k-m)!} \, \frac{(-t)^m}{m!} , \quad k\in \mathbb{N}_0 ,
$$
where $\mathbb{N}_0=\{0\}\cup\mathbb{N}$. Since
$\int_0^\infty (L_k^{(\alpha)})^2 w\, dt = \Gamma(k+\alpha+1)/k!$, the system orthonormal in
$L_{2,w}$ is
\begin{equation}\label{normalisation}
\ell_k^{(\alpha)}(t) := \left(\frac{k!}{\Gamma(k+\alpha+1)}\right)^{1/2} L_k^{(\alpha)}(t) .
\end{equation}
We call $\varphi_k^{(\alpha)}(t) = \ell_k^{(\alpha)}(t) \sqrt{w(t)}$ the Laguerre functions, and
the Fourier--Laguerre coefficients of a function $f$ --- the numbers
$\langle f, \ell_{k}^{(\alpha)}\rangle$.

By $C$ we denote the space of all real-valued functions continuous on
$(0,\infty)$, and we put $\|g\|_{C}:= \sup_{t > 0} |g(t)|\in[0,+\infty]$. This quantity
is defined for every $g\in C$, but it is not a norm on the whole of $C$: it can be
infinite. Everywhere below, wherever it occurs, it is applied to a product of the form $g\sqrt{w}$, and
its finiteness is part of the statement being proved, not a convention of notation. The space
$C$ itself is taken in full where it serves as the range of the recovery operator (section
\ref{opt}): the value of such an operator can be a polynomial unbounded on $(0,\infty)$.
By $\ell_p$, $1\leq p\leq\infty$, we denote the space
of numerical sequences $\mathbf{x}=\{x_{k}\}_{k\in\mathbb{N}_0}$ with the usual norm
$$
\|\mathbf{x}\|_{\ell_p}  := \left\{
\begin{array}{cl}
\bigg(\sum\limits_{k\in\mathbb{N}_0} |x_{k}|^p\bigg)^{\frac{1}{p}} ,
 \ & 1\leq p<\infty ,
\\\\
\sup\limits_{k\in\mathbb{N}_0}  |x_{k}| ,
  \ & p=\infty .
\end{array}
\right.
$$
The function classes are the weighted Wiener classes (see, for example, \cite{Kolom2023}):
$$
W_{s}^\mu =\{f\in L_{2,w}: \ \|f\|_{s,\mu}^s :=\sum_{k=0}^{\infty} ({\underline{k}})^{s\mu}|\langle f,
\ell_{k}^{(\alpha)}\rangle|^s<\infty\},\quad 1\le s<\infty,
$$
$$
W_{\infty}^\mu =\{f\in L_{2,w}:
 \|f\|_{\infty,\mu} =
 \sup\limits_{k\in \mathbb{N}_0}\, \underline{k}^{\mu}\,
 |\langle f , \ell_{k}^{(\alpha)} \rangle | < \infty\},
$$
where $\mu>0$ and $\underline{k}=\max\{1,k\}$. We use the same notation for
a space and its unit ball, which we call a function class.

\vskip 4mm

\begin{remark}\label{Tom}
In \cite{Tom}, it was established that the spaces $W^r_{2}$, $r\in\mathbb{N}$,
consist of functions $f$, for which
$f, f',\ldots, f^{(r-1)}$
are absolutely continuous on any interval $(\eta_1,\eta_2)$, $0<\eta_1<\eta_2<\infty$, and the condition
$$\int_{0}^{\infty}  |f^{(r)}(t)|^2 t^{r+\alpha} e^{-t} \mathrm{d} t < \infty $$
is satisfied,
therefore, $W^r_{2}$ can be equipped with an equivalent norm
$$
\|f\|_{W^r_{2}} := \left(\sum_{j=0}^r \int_{0}^{\infty}  |f^{(j)}(t)|^2 t^{j+\alpha} e^{-t}
\mathrm{d}t\right)^{1/2} .
$$
\end{remark}

\vskip 2mm

We assume that instead of the exact Fourier--Laguerre coefficients
$\mathbf{f} = ( \langle f, \ell^{(\alpha)}_{k}\rangle )_{k\in\mathbb{N}_0}$ we are given
perturbed data
$$
\mathbf{f}^\delta = \mathbf{f} + \bxi , \qquad
\mathbf{f}^\delta = ( f^\delta_k )_{k \in \mathbb{N}_0} ,
$$
where the error sequence $\bxi=(\xi_k)_{k\in\mathbb{N}_0}$ is a sequence
of real numbers of which it is not assumed that it is the sequence
of coefficients of any element of $L_{2,w}$, and the error of the input data is measured in
the $\ell_p$ norm: for some $1 \leq p \leq \infty$
\begin{equation}\label{perturbation2}
	\|\bxi\|_{\ell_p} \leq \delta, \quad 0 < \delta < 1.
\end{equation}
We study the regularizing properties of the summation methods
\begin{equation}\label{ModVer}
S_N^\nu \mathbf{f}^\delta (t) = \sum_{k=0}^{N}
\nu_k^N f^\delta_k \, \ell^{(\alpha)}_{k}(t)
\end{equation}
for triangular arrays $\nu=\{\nu_k^N\}$, $k=0,\ldots,N$, $N\in \mathbb{N}$, subject to the
condition: there exist $C(\nu)>0$ and $\theta>0$ such that
\begin{equation}\label{qual}
|1- \nu_k^N| \le C(\nu) \left(\frac{k}{N}\right)^\theta,\quad 0^\theta=0, \quad  k=0,1,\ldots,N .
\end{equation}
By $\mathcal{S}^\theta$ we denote the class of all methods (\ref{ModVer}) satisfying
(\ref{qual}).
Of greatest interest are the methods (\ref{ModVer}) with the maximal admissible value of $\theta$, i.e.\ with $\theta=\infty$.
It is exactly such methods that can be effectively applied to derivatives of any order.
Among the methods (\ref{ModVer}) with maximal $\theta$, the best known are the Fourier and the de la Vall\'{e}e Poussin methods.
Below we describe two more truncation methods with an arbitrary $\theta$, taken by us from the work \cite{OcTh}.

\begin{example}\label{Ex1}\rm
\emph{Smoothed truncation methods.} Let $\rho\in(0,1)$ be fixed and $n_\rho:=\lfloor\rho
N\rfloor$; the method leaves the indices $k\le n_\rho$ unchanged, and on the remaining ones it
damps smoothly.

{\rm(1)} \emph{Linear truncation} (the de la Vall\'{e}e Poussin method):
\begin{equation}\label{vp}
\nu^N_k =
\begin{cases}
1 , & 0\le k\le n_\rho , \\[2pt]
\dfrac{N-k}{N-n_\rho} , & n_\rho<k\le N .
\end{cases}
\end{equation}
For $\rho=1/2$ and even $N=2n$ this is exactly the classical de la Vall\'{e}e Poussin method
(see, for example, \cite[example 1.1]{SolVol}):
$\nu^{2n}_k=1$ for $k\le n$ and $\nu^{2n}_k=(2n-k)/n$ for $n<k\le 2n$.

{\rm(2)} \emph{Cosine truncation} (the Hann window):
\begin{equation}\label{hann}
\nu^N_k =
\begin{cases}
1 , & 0\le k\le n_\rho , \\[2pt]
\dfrac12\left(1+\cos\dfrac{\pi(k-n_\rho)}{N-n_\rho}\right) , & n_\rho<k\le N .
\end{cases}
\end{equation}

Both methods are given by an explicit formula, computable from $k$ in a bounded number of operations, and both
satisfy (\ref{qual}) \emph{for every} $\theta>0$ with $C(\nu)=\rho^{-\theta}$.
\end{example}

By $c$ we denote a positive
constant depending on $\alpha$, $r$, $\mu$, $s$, $p$, $\theta$, $C(\nu)$ and (in sections \ref{opt},
\ref{sharp}) on the constant $K$ of the definition (\ref{Nmin}), but never on
$N$, $\delta$, $f$, $\mathbf{f}^\delta$ or $\bxi$; its value may change from line to
line. By $S_N f$ we denote the ordinary Fourier--Laguerre sum, that is, (\ref{ModVer}) with
$\nu_k^N\equiv1$ and exact data.

For nonnegative quantities $A$ and $B$ depending on the variables $N$, $k$, $t$ or $\delta$,
the notation $A\preceq B$ means that $A\le c\,B$ with a constant $c$ of the indicated kind, that is, not
depending on these variables; $A\succeq B$ means $B\preceq A$, and $A\asymp B$ --- that
both hold. The sign $\asymp$ is applied also to the rule for choosing the truncation level: the notation
$N\asymp\delta^{-1/(\mu+1/s-1/p)}$ means that $a_1\delta^{-1/(\mu+1/s-1/p)}\le N\le a_2\delta^{-1/(\mu+1/s-1/p)}$ with
constants $0<a_1\le a_2$ not depending on $\delta$.

Thus, it is $f^{(r)}$ that is to be recovered, and the error is measured in the weighted uniform
metric:
$$
\bigl\| \bigl(f-S_N^\nu \mathbf{f}^\delta\bigr)^{(r)} \sqrt{w} \bigr\|_{C} .
$$
Before estimating this quantity, three auxiliary facts are needed: the differentiation
formula, the growth of a single differentiated Laguerre function, and the condition under
which this quantity is finite at all on the class.

\subsection{Differentiation formula}

\begin{lemma}\label{lemma_diff}
For all integers $k\ge r$ and all $t>0$ there holds the equality
\begin{equation}\label{diff_formula}
\frac{d^r}{dt^r}\,\ell_k^{(\alpha)}(t)
= (-1)^r \sqrt{\frac{k!}{(k-r)!}}\; \ell_{k-r}^{(\alpha+r)}(t) ,
\end{equation}
moreover, $\frac{d^r}{dt^r}\ell_k^{(\alpha)}\equiv 0$ for $k<r$.
\end{lemma}

\textit{Proof.} It is classical (see, for example, \cite[ch.\,5]{Szego})
that
$$
\frac{d}{dt}L_k^{(\alpha)}=-L_{k-1}^{(\alpha+1)},
$$
whence
$$
\frac{d^r}{dt^r}L_k^{(\alpha)}=(-1)^r L_{k-r}^{(\alpha+r)}
$$
for $k\ge r$; for $k<r$ the derivative is identically zero, since the degree of
$L_k^{(\alpha)}$ equals $k$.
By virtue of $\Gamma(k+\alpha+1)=\Gamma((k-r)+(\alpha+r)+1)$ we have
$$
\left(\frac{k!}{\Gamma(k+\alpha+1)}\right)^{1/2}
= \left(\frac{k!}{(k-r)!}\right)^{1/2}
\left(\frac{(k-r)!}{\Gamma((k-r)+(\alpha+r)+1)}\right)^{1/2} .
$$
Whence, taking into account (\ref{normalisation}), we obtain (\ref{diff_formula}). \\
$\Box$

\vskip 2mm

Formula (\ref{diff_formula}) together with $w(t)=t^\alpha e^{-t}$ gives the identity that we
use everywhere below:
\begin{equation}\label{reduction}
\bigl(\ell_k^{(\alpha)}\bigr)^{(r)}(t)\sqrt{w(t)}
= (-1)^r \sqrt{\frac{k!}{(k-r)!}}\; t^{-r/2}\, \varphi_{k-r}^{(\alpha+r)}(t) ,
\qquad k \ge r .
\end{equation}
Both factors on the right grow in $k$ no faster than $k^{r/2}$; this is exactly the content of
the following lemma.

\subsection{Scale of the derivative}


\begin{lemma}\label{lemma_scale}
Let $\alpha\ge0$ be real and $j$ --- an integer, $1\le j\le r$. Then for all $k\ge j$
$$
\bigl\| \bigl(\ell_k^{(\alpha)}\bigr)^{(j)} \sqrt{w} \bigr\|_{C} \le c\, k^{j} .
$$
\end{lemma}

\textit{Proof.} The derivation of (\ref{diff_formula}) and (\ref{reduction}) uses nothing
but the integrality of the order of differentiation, so both identities hold with $j$ in place of $r$, and
below they are applied in this form. Put $m=k-j$, $\beta=\alpha+j$ and $\nu=4m+2\beta+2$, so that
$\nu\le c\,k$. From the pointwise estimate \cite[(2.5), p.\,435]{Muck70_2}, valid for
$\beta\ge0$ and all $m\ge0$, it follows that
$$
|\varphi_m^{(\beta)}(t)|\le C\,(t\nu)^{\beta/2}\ \text{for } 0<t\le1/\nu ,
\qquad
|\varphi_m^{(\beta)}(t)|\le C\ \text{for } t>1/\nu ,
$$
where $C$ does not depend on $t$ and $m$. For $0<t\le1/\nu$ this gives
$$t^{-j/2}|\varphi_m^{(\beta)}(t)|\le C\,t^{\alpha/2}\nu^{\beta/2}\le C\,\nu^{j/2},$$
while for
$t>1/\nu$ we have
$$t^{-j/2}|\varphi_m^{(\beta)}(t)|\le C\,\nu^{j/2}.$$
Hence,
$$\sup_{t>0}t^{-j/2}|\varphi_m^{(\beta)}(t)|\le c\,k^{j/2},$$
and (\ref{reduction}) together with $k!/(k-j)!\le k^{j}$ gives the statement of the lemma.\\
$\Box$

\vskip 2mm

\begin{remark}\label{rem_scale_known}\rm
An estimate of the form $\|(L_k^{(\alpha)})^{(j)}\|_\infty \le c\, k^{j}$ was obtained earlier in
\cite[lemma 1]{Satake} and \cite[(3.4)]{Plewa} under the restriction $j\le[\alpha/2]$; lemma
\ref{lemma_scale} does not require this restriction. The order $k^{j}$ is sharp: it follows also from
the lower bound \cite[(6.12), p.\,453]{Muck70_2} together with (\ref{reduction}).
\end{remark}

\section{Error estimate}\label{EE}

The condition $\mu>r+1-1/s$ means exactly $\mu+1/s-1-r>0$. Below we shall see that this quantity
is the exponent of accuracy, while $\mu+1/s-1/p$ is the denominator of the balance rule. Let us write
\begin{equation}\label{fullError}
\bigl(f-S_N^\nu \mathbf{f}^\delta\bigr)^{(r)} = \bigl(f-S_N f\bigr)^{(r)}+
\bigl(S_N f-S_N^{\nu}f\bigr)^{(r)}+\bigl(S_N^{\nu} f-S_N^{\nu} \mathbf{f}^\delta\bigr)^{(r)} ,
\end{equation}
and estimate in the metric $C$ each of the three differences multiplied by $\sqrt{w}$.

\vskip 2mm

\begin{lemma}\label{lemma_BoundT1}
Let $f\in W^\mu_{s}$, $1\leq s \le \infty$, $\mu>r+1-1/s$. Then for the ordinary
Fourier--Laguerre sum and every $N\ge r$
$$
\bigl\|(f-S_N f)^{(r)}\sqrt{w}\bigr\|_{C}\leq c\,\|f\|_{s,\mu}\, N^{-(\mu+1/s-1-r)} .
$$
\end{lemma}

\textit{Proof.} We first show that the Fourier--Laguerre series of a function $f$ can be
differentiated termwise $r$ times. For $1\le j\le r$ and $k\ge j$, by lemma \ref{lemma_scale},
$\|(\ell^{(\alpha)}_k)^{(j)}\sqrt{w}\|_C\le c\,k^{j}\le c\,\underline{k}^{\,r}$ (the terms with $k<j$
are zero by lemma \ref{lemma_diff}), and for $j=0$
$\|\varphi^{(\alpha)}_k\|_C\le C$ by \cite[(2.5)]{Muck70_2}. Therefore each of the series
$\sum_k\langle f,\ell^{(\alpha)}_k\rangle(\ell^{(\alpha)}_k)^{(j)}\sqrt{w}$, $0\le j\le r$,
is majorized by the numerical series $c\sum_k\underline{k}^{\,r}|\langle f,\ell^{(\alpha)}_k\rangle|$,
which for $\mu>r+1-1/s$ converges by the H\"{o}lder inequality (the same computation as below). By
the Weierstrass test the series $\sum_k\langle f,\ell^{(\alpha)}_k\rangle(\ell^{(\alpha)}_k)^{(j)}$
converge uniformly on every interval $[a,b]\subset(0,\infty)$, where $\sqrt{w}$ is bounded away from
zero. Therefore the sum of the series $\sum_k\langle f,\ell^{(\alpha)}_k\rangle\ell^{(\alpha)}_k$ is a function
of class $C^{r}$ on $(0,\infty)$; and since the system $\{\ell^{(\alpha)}_k\}$ is complete in $L_{2,w}$
\cite[theorem 5.7.1]{Szego}, the partial sums $S_N f$ converge to $f$ in $L_{2,w}$, and the limit coincides with $f$
almost everywhere. Hence $f$ has on $(0,\infty)$ a representative of class $C^{r}$,
$f^{(r)}=\sum_{k\ge r}\langle f,\ell^{(\alpha)}_k\rangle(\ell^{(\alpha)}_k)^{(r)}$, and by lemma
\ref{lemma_scale}
$$
\bigl\|(f-S_N f)^{(r)} \sqrt{w}\bigr\|_C
= \Bigl\|\sum_{k>N} \langle f, \ell^{(\alpha)}_{k}\rangle
\bigl(\ell^{(\alpha)}_{k}\bigr)^{(r)}\sqrt{w}\Bigr\|_C
\le c \sum_{k>N} k^{r}\,|\langle f, \ell^{(\alpha)}_{k}\rangle | .
$$
For $1<s<\infty$ the H\"{o}lder inequality and the definition of $W^\mu_s$ give
$$
\sum_{k>N} k^{r}|\langle f, \ell^{(\alpha)}_{k}\rangle |
\le \Bigl(\sum_{k>N} k^{\mu s}|\langle f, \ell^{(\alpha)}_{k}\rangle |^s\Bigr)^{1/s}
\Bigl(\sum_{k>N} k^{(r-\mu)s/(s-1)}\Bigr)^{(s-1)/s} ,
$$
$$
\Bigl(\sum_{k>N} k^{(r-\mu)s/(s-1)}\Bigr)^{(s-1)/s}
\le \frac{N^{r-\mu+(s-1)/s}}{\bigl((\mu-r)s/(s-1)-1\bigr)^{(s-1)/s}} ,
$$
and moreover the last series converges, since $(\mu-r)s/(s-1)>1$.
For $s=1$ one uses the estimate
$\sum_{k>N} k^{r-\mu}|k^{\mu}\langle f,\ell^{(\alpha)}_k\rangle| \le N^{r-\mu}\|f\|_{1,\mu}$,
and for $s=\infty$ --- comparison with the integral
$\sum_{k>N}k^{r-\mu}\le N^{r-\mu+1}/(\mu-r-1)$.\\
$\Box$

\vskip 2mm

\begin{lemma}\label{lemma_BoundT2}
Let $f\in W^\mu_{s}$, $1\leq s \le \infty$, $\mu>0$. Then for every fixed
$\theta>\mu+1/s-1-r$, for any method from $\mathcal{S}^\theta$ and for \emph{every} $N\ge r$
$$
\bigl\|(S_N f-S^{\nu}_N f)^{(r)}\sqrt{w}\bigr\|_{C}\leq c\,\|f\|_{s,\mu}\, N^{-(\mu+1/s-1-r)} .
$$
\end{lemma}

\textit{Proof.} On the left-hand side there stands an $r$ times differentiated polynomial of degree
not exceeding $N$, so the sum below is finite and no question of convergence arises. By lemma \ref{lemma_scale} and (\ref{qual})
$$
\begin{aligned}
\bigl\|(S_N f-S^{\nu}_N f)^{(r)}\sqrt{w}\bigr\|_C
&\le \sum_{k=r}^{N}\left|1-\nu_k^N\right|
 |\langle f,\ell^{(\alpha)}_{k}\rangle|
 \bigl\|\bigl(\ell_k^{(\alpha)}\bigr)^{(r)}\sqrt{w}\bigr\|_{C} \\
&\le c\,C(\nu)\,N^{-\theta}\sum_{k=r}^{N}k^{\theta+r}
 |\langle f,\ell^{(\alpha)}_{k}\rangle| .
\end{aligned}
$$
For $1<s<\infty$ the H\"{o}lder inequality gives
$$
\le c\, N^{-\theta} \Bigl(\sum_{k=r}^{N} k^{\mu s}|\langle f, \ell^{(\alpha)}_{k}\rangle |^s\Bigr)^{1/s}
\Bigl(\sum_{k=r}^{N} k^{(\theta+r-\mu)s/(s-1)}\Bigr)^{(s-1)/s}
\le c\,\|f\|_{s,\mu}\, N^{-\theta}\,N^{\theta+r-\mu+(s-1)/s} ,
$$
where at the last step $(\theta+r-\mu)s/(s-1)>-1$ was used. For $s=1$ one uses
$\sup_{r\le k\le N}k^{\theta+r-\mu}=N^{\theta+r-\mu}$, and for $s=\infty$ --- the relation
$\sum_{k=r}^{N}k^{\theta+r-\mu}\le c\,N^{\theta+r-\mu+1}$, valid for $\theta+r-\mu>-1$.\\
$\Box$

\vskip 2mm

\begin{lemma}\label{lemma_BoundT3}
Let $1\leq p\le \infty$ and let (\ref{perturbation2}) hold. Then for any method from
$\mathcal{S}^\theta$, $\theta>0$, and every $N\ge r$
$$
\bigl\|(S^{\nu}_N f-S^{\nu}_N \mathbf{f}^\delta)^{(r)}\sqrt{w}\bigr\|_{C}\leq c\, \delta\, N^{r+1-1/p} .
$$
\end{lemma}

\textit{Proof.} By virtue of (\ref{qual}) the array $\nu$ is uniformly bounded,
$|\nu^N_k|\le 1+C(\nu)$, so that by lemma \ref{lemma_scale}
$$
\bigl\|(S^{\nu}_N f-S^{\nu}_N \mathbf{f}^\delta)^{(r)} \sqrt{w}\bigr\|_C
\le c\,(1+C(\nu)) \sum_{k=r}^{N} k^{r}|\xi_k| .
$$
For $1<p<\infty$ from the H\"{o}lder inequality we obtain
$$
\sum_{k\le N} k^{r}|\xi_k| \le \|\bxi\|_{\ell_p}
\Bigl(\sum_{k\le N} k^{rp/(p-1)}\Bigr)^{(p-1)/p}
\le c\, \delta\; N^{\,r+1-1/p} ,
$$
where at the last step there was used $\sum_{k\le N}k^{\lambda}\le\frac{\lambda+2}{\lambda+1}
N^{\lambda+1}$, valid for all $\lambda>-1$ and $N\ge1$, --- the same inequality on a partial
sum as in the proof of lemma \ref{lemma_BoundT2}.
For $p=1$ the same estimate follows immediately from
$\sum_{k\le N}k^{r}|\xi_k|\le N^{r}\|\bxi\|_{\ell_1}\le\delta N^{r}$, and
for $p=\infty$ --- from
$\sum_{k\le N}k^{r}|\xi_k|\le\|\bxi\|_{\ell_\infty}\sum_{k\le N}k^{r}\le c\,\delta N^{r+1}$.\\
$\Box$

\vskip 2mm

\begin{remark}\label{rem_noise_sharp}\rm
The estimate of lemma \ref{lemma_BoundT3} is sharp in order already for the Fourier method $\nu^N_k\equiv1$.
Indeed, let $r\ge1$, $N\ge4r$, $m=\lfloor(N-r)/3\rfloor$, $\Omega\subset[m+r,3m+r]$,
$\card(\Omega)=m$, and let $\xi_k=\delta\,m^{-1/p}$ for $k\in\Omega$ and $\xi_k=0$ for $k\notin\Omega$;
then $\|\bxi\|_{\ell_p}=\delta$, and by lemma \ref{lemma_BE} of section \ref{sharp} we have
$\bigl\|\sum_{k=0}^{N}\xi_k(\ell^{(\alpha)}_k)^{(r)}\sqrt{w}\bigr\|_C\ge\bar c\,\delta\,m^{\,r+1-1/p}
\ge c\,\delta\,N^{\,r+1-1/p}$.
\end{remark}

\vskip 2mm

\begin{theorem}\label{Th_up}
Let $f\in W^\mu_{s}$, $\|f\|_{s,\mu}\le1$, $1\leq s \le \infty$, $\mu>r+1-1/s$, and let
(\ref{perturbation2}) hold for some $1\le p\le \infty$. Then for every fixed
$\theta>\mu+1/s-1-r$, for any method $S^{\nu}_N$ from $\mathcal{S}^\theta$, every $0<\delta<1$ and
every integer truncation level $N\ge r$ subject to the relation
$$
N \asymp \delta^{-1/(\mu+1/s-1/p)} ,
$$
there holds the estimate
$$
\bigl\|(f-S^{\nu}_N \mathbf{f}^\delta)^{(r)} \sqrt{w}\bigr\|_{C} \leq c\;
\delta^{\,(\mu+1/s-1-r)/(\mu+1/s-1/p)}\, .
$$
\end{theorem}

\textit{Proof.} The condition $N\ge r$ is exactly the domain of applicability of lemmas
\ref{lemma_BoundT1}, \ref{lemma_BoundT2} and \ref{lemma_BoundT3}, so all three are applicable, and
(\ref{fullError}) gives
$$
\bigl\|(f-S^{\nu}_N \mathbf{f}^\delta)^{(r)} \sqrt{w}\bigr\|_{C}
\le c\, N^{-(\mu+1/s-1-r)} + c\, \delta N^{\,r+1-1/p}
= c\, N^{\,r+1}\bigl(N^{-\mu-1/s}+\delta N^{-1/p}\bigr) .
$$
Both terms on the right-hand side of the last relation
are balanced for $N\asymp\delta^{-1/(\mu+1/s-1/p)}$.
This is what had to be proved.\\
$\Box$

\vskip 2mm

\begin{corollary}\label{Cor1}
Let $\alpha\ge0$, $1\le s\le\infty$, $1\le p\le\infty$, $\mu>r+1-1/s$ and
$\theta>\mu+1/s-1-r$ be fixed. Then every summation method (\ref{ModVer}) from
$\mathcal{S}^\theta$, at every integer truncation level $N\ge r$ with
$N\asymp\delta^{-1/(\mu+1/s-1/p)}$ and for every
$0<\delta<1$, recovers the $r$-th derivative of an arbitrary function
$f\in W^{\mu}_{s}$ in the weighted uniform metric with accuracy
$$
O\Bigl(\delta^{\frac{\mu+1/s-1-r}{\mu+1/s-1/p}}\Bigr) ,
$$
using
$$
\card(\{r,\ldots,N\}) \asymp N \asymp \delta^{-\frac{1}{\mu+1/s-1/p}}
$$
perturbed Fourier--Laguerre coefficients. The constant in $O(\cdot)$ is the constant $c$
of theorem \ref{Th_up}: it depends on $\alpha$, $r$, $\mu$, $s$, $p$, $\theta$, $C(\nu)$ and on
the constants of the relation $N\asymp\delta^{-1/(\mu+1/s-1/p)}$, but not on $f$, not on $\delta$ and not on $N$.
\end{corollary}

\vskip 2mm

\begin{remark}\label{rem_upper_known}\rm
The scheme of obtaining an accuracy estimate through a two-term majorant of the error,
an estimate of the noise contribution and the balance rule was known earlier, (see, for example, \cite{Mathe&Per}, \cite{Sol_Stas_UMZ2022}).
Lemma \ref{lemma_BoundT3} is a general estimate of noise propagation for an orthonormal system
satisfying $\|\varphi_k\|_C\le c\,k^{\gamma}$, applied with $\gamma=r$.
What is new here is the exponent with which lemma \ref{lemma_BoundT3} is applied. Each
differentiation of a Chebyshev or Legendre polynomial costs a factor $k^{2}$, whereas by
lemma \ref{lemma_scale} the weighted Laguerre system loses, per single differentiation, only
the factor $k$. Therefore in the numerator of the exponent of accuracy of theorem \ref{Th_up} there stands $-r$ where
in the exponents of the works \cite{Sem_Sol_2026} and \cite{Kys} there stands $-2r$.
\end{remark}

\vskip 2mm

\begin{remark}\label{rem_level}\rm
Differentiation spoils the accuracy and does not change the number of coefficients $N$
that has to be taken: one and the same truncated sum recovers both the function
and all its derivatives up to the order allowed by the class.
The noted effect is especially useful in composing algorithms for solving problems
that operate with derivatives of different orders. In this case there arises no problem of choosing
a compromise $N$, which could lead to a loss of the optimal order of accuracy.
\end{remark}

\vskip 2mm

\section{Optimal recovery and information complexity. Problem setting}\label{opt}

Our next goal is to show that the methods (\ref{ModVer}) from $\mathcal{S}^\theta$ attain
order-optimal accuracy estimates in the recovery of $f^{(r)}$, while for $p=2$ this estimate is
improved by no mapping whatsoever.

By the \emph{information operator} $G:L_{2,w}\to\ell_2$ we mean the operator assigning to a
function $f$ the sequence of its Fourier--Laguerre coefficients, $Gf=\mathbf{f}$; such data
are called Galerkin information (see, for example, \cite{Wer87}). The operator $G$ describes
exact information, whereas the available data are perturbed: $\mathbf{f}^\delta=Gf+\bxi$, where
$\|\bxi\|_{\ell_p}\le\delta$. Since $\{\ell^{(\alpha)}_k\}$ is an
orthonormal basis of $L_{2,w}$, the exact data (the image of the operator $G$)
run over the whole of $\ell_2$, while the perturbed ones lie in the space
$\ell_2+\ell_p=\ell_{\max\{p,2\}}$, which is what we take as the domain of the recovery
operator. For $p=\infty$ not all of its elements are attainable --- for example, the constant
sequence $y_k\equiv1$ is not representable in the form $Gf+\bxi$ for $\delta<1$; this does not
affect the quantities studied below, since the error of a method reads the recovery operator
only on attainable data, and on the remaining ones it may be defined arbitrarily.

By a \emph{recovery operator} we mean an arbitrary mapping
$\psi:\ell_{\max\{p,2\}}\to C$, and by an \emph{approximation method} --- an arbitrary operator
$\psi G$, understood as giving an approximation to $f^{(r)}$ and not to $f$. The set of all such
operators is denoted by $\Psi$; of a recovery operator neither linearity,
nor continuity, nor stability is required. The error of a method on the class is
$$
e_{\delta}\bigl(W_{s}^\mu, \psi G, C, \ell_p\bigr) =
\sup_{f \in W_{s}^\mu}\ \sup_{\|\bxi\|_{\ell_p} \leq \delta}\
\bigl\| \bigl(f^{(r)} - \psi (G f + \bxi)\bigr) \sqrt{w} \bigr\|_{C} ,
$$
and we study the quantity
$$
\mathcal{E}^r_\delta \bigl(W_{s}^\mu, \Psi, C, \ell_p\bigr) =
\inf_{\psi G \in \Psi} e_{\delta}\bigl(W_{s}^\mu, \psi G, C, \ell_p\bigr) .
$$
For a finite $\Omega\subset\mathbb{N}_0$ we put $G_\Omega:=P_\Omega G$, where $P_\Omega$ is the
coordinate projection retaining the elements with indices from $\Omega$; by a \emph{numerical
differentiation algorithm} we mean an arbitrary operator $\psi G_\Omega$, and the set
of all such operators with $\card(\Omega)\le N$ --- by $\Psi_N\subset\Psi$.
The \emph{minimal radius of Galerkin information} is
$$
R^r_{N,\delta}\bigl(W^{\mu}_{s}, \Psi_N, C, \ell_p\bigr) =
\inf_{\psi G_\Omega \in \Psi_N} e_\delta\bigl(W^{\mu}_{s}, \psi G_\Omega, C, \ell_p\bigr) ,
$$
moreover the inclusion $\Psi_N\subset\Psi$ (to the algorithm $\psi G_\Omega$ there corresponds the recovery operator
$\psi\circ P_\Omega$, whose error is no larger) gives $\mathcal{E}^r_\delta \le R^r_{N,\delta}$ for every $N$
and every $\delta>0$ --- and at once for all $1\le p\le\infty$: the parameter $p$ does not enter
this inclusion, while it enters both quantities in the same way, both the data space and the constraint on the noise
(\ref{perturbation2}).

In addition, we are interested in the smallest amount of information that is already sufficient for
order-optimal recovery. Let us fix once and for all a constant $K>0$ and put
\begin{equation}\label{Nmin}
N_{\min}=N_{\min}(\delta)=\min\Bigl\{N\in\mathbb{N}:\
R^r_{N,\delta}\bigl(W^{\mu}_{s},\Psi_N,C,\ell_p\bigr)\le K\,\delta^{\,(\mu+1/s-1-r)/(\mu+1/s-1/p)}\Bigr\} ,
\end{equation}
that is, the smallest $N$ at which numerical differentiation reaches --- up to the
constant $K$ --- the accuracy $\delta^{(\mu+1/s-1-r)/(\mu+1/s-1/p)}$ of corollary \ref{Cor1}, order-optimal by
virtue of theorem \ref{Th_R_order} of section \ref{sharp} for all $0<\delta<1$. The constant $K$ is taken no smaller than the constant $C_0$
of corollary \ref{Cor_Nrange}{\rm (ii)} with $c_0=1$; in theorem \ref{Th_Nmin} it is shown that then
the set on the right-hand side is nonempty for every $0<\delta<1$. The quantity
$N_{\min}$ is called the \emph{information complexity} of numerical differentiation.
The closest predecessor of this quantity known to us is \cite{MagOs}, where for optimal
recovery of functions and their derivatives on compact manifolds from Fourier coefficients
given with an error the smallest number of the first Fourier coefficients needed for the most
precise recovery is found exactly.

It remains to bring into consideration one more minimax quantity.
Unlike all the quantities given above, which operate with methods
using the \emph{coefficients}, it has as its input the perturbed function itself,
while the recovering mapping is arbitrary, --- for $p=2$ it coincides with $\mathcal{E}^r_\delta$.
Let us prove this fact in full: for from this it follows that the order of optimal
recovery found below is not improved by any mapping whatsoever, and not only by methods of the kind considered.

Denote by $\mathcal{M}$ the set of \emph{all} mappings $\varphi:L_{2,w}\to C$ ---
with no requirements of linearity, continuity or stability, --- understood as giving
an approximation to $f^{(r)}$, and put
\begin{equation}\label{Edelta}
E^r_{\delta}\bigl(W^{\mu}_{s},\mathcal{M},C\bigr) =
\inf_{\varphi\in\mathcal{M}}\ \sup_{f\in W^{\mu}_{s}}\
\sup_{\substack{h\in L_{2,w}\\ \|f-h\|_{2}\le\delta}}\
\bigl\|\bigl(f^{(r)}-\varphi(h)\bigr)\sqrt{w}\bigr\|_{C} .
\end{equation}
The derivative stands here on $f$ \emph{before} the optimization: for $r\ge1$ the operator $d^{r}/dt^{r}$
is unbounded from $L_{2,w}$ into $C$, so that the smallness of $\|f-\varphi(h)\|_2$ by itself gives
no estimate whatever of the error after differentiation.

\begin{proposition}\label{prop_MPsi}
Let $p=2$ and $\mu>r+1-1/s$. Then the mapping $J:\psi\mapsto\varphi:=\psi\circ G$ is a bijection of the set
of recovery operators defining $\Psi$ onto $\mathcal{M}$, and it preserves the error:
for every $\psi$ and every $\delta>0$
$$
e_{\delta}\bigl(W^{\mu}_{s},\psi G,C,\ell_2\bigr)
= \sup_{f\in W^{\mu}_{s}}\ \sup_{\substack{h\in L_{2,w}\\ \|f-h\|_{2}\le\delta}}\
\bigl\|\bigl(f^{(r)}-(J\psi)(h)\bigr)\sqrt{w}\bigr\|_{C} .
$$
In particular,
\begin{equation}\label{ME}
E^r_{\delta}\bigl(W^{\mu}_{s},\mathcal{M},C\bigr)
= \mathcal{E}^r_\delta\bigl(W^{\mu}_{s},\Psi,C,\ell_2\bigr) .
\end{equation}
\end{proposition}

\textit{Proof.} Since $\{\ell^{(\alpha)}_k\}_{k\in\mathbb{N}_0}$ is an orthonormal
basis of $L_{2,w}$, the operator $G$, by the Parseval identity and the Riesz--Fischer theorem, is an isometry
of $L_{2,w}$ onto the whole of $\ell_2$; the inverse operator $G^{-1}$ is an isometry as well. For $p=2$ the domain
of the recovery operator is $\ell_{\max\{p,2\}}=\ell_2$, that is, exactly
the image of $G$, so that $\psi\circ G$ is a mapping $L_{2,w}\to C$, an element of $\mathcal{M}$, while
$\varphi\mapsto\varphi\circ G^{-1}$ is inverse to $J$, since $(\psi\circ G)\circ G^{-1}=\psi$ and
$(\varphi\circ G^{-1})\circ G=\varphi$. Hence $J$ is bijective: no requirements other than
that of being mappings are imposed on their elements either by $\Psi$ or by $\mathcal{M}$.

Let $\varphi=J\psi$ and $f\in W^{\mu}_{s}$. By the isometry, the substitution $\bxi:=Gh-Gf$ maps
the ball $\|f-h\|_{2}\le\delta$ one-to-one onto the ball $\|\bxi\|_{\ell_2}\le\delta$ and gives
$\varphi(h)=\psi(Gh)=\psi(Gf+\bxi)$: the suprema on both sides are taken of one and the same quantity over
one and the same set and are equal for each $f$, and with them --- the suprema over $f\in W^{\mu}_{s}$ as well.
Finally, $J$ carries the set of recovery operators defining $\Psi$ onto the whole of
$\mathcal{M}$ without changing the error, so that the infimum over $\psi G\in\Psi$ and the infimum over
$\varphi\in\mathcal{M}$ are infima of one and the same numerical set; this is precisely (\ref{ME}). \\
$\Box$
\vskip 2mm

Thus we have established that for $p=2$ \emph{no} mapping
of the spaces under consideration --- linear or not, stable or not --- recovers
$f^{(r)}$ with accuracy of a better order than the methods using Galerkin information
as input data.

\section{Sharp orders of optimal recovery and information complexity}\label{sharp}

The lower estimates of this section rest on the following block estimate, which replaces
\cite[lemma 4.1]{SolVol}.

The proof of the lemma involves the modified Bessel function of the first
kind $I_\nu$, defined by the series
\begin{equation}\label{bessel}
I_\nu(z) = \left(\frac z2\right)^{\nu}
\sum_{j=0}^{\infty}\frac{(z^2/4)^{j}}{j!\,\Gamma(\nu+j+1)}
\end{equation}
(\cite[9.6.10]{AbrSt}). For $\nu=0$, by virtue of $\Gamma(j+1)=j!$ there holds
$$
I_0(z)=\sum_{j=0}^{\infty}\frac{(z^2/4)^{j}}{(j!)^{2}} ,
\qquad\text{which is equivalent to}\qquad
\sum_{j=0}^{\infty}\frac{x^{j}}{(j!)^{2}}=I_0\bigl(2\sqrt{x}\bigr) ,
\quad x\ge0 .
$$

\begin{lemma}\label{lemma_BE}
Let $\alpha\ge0$ be real, $N\ge1$ an integer, $1\le r\le N$, and let
$$
f(t)=\sum_{k\in \Omega} \ell^{(\alpha)}_k(t) .
$$
Then at the point $t_*=1/(12N)$, for every integer $k$ with $N+r\le k\le3N+r$, setting $m=k-r$,
$$
(-1)^r\bigl(\ell^{(\alpha)}_k\bigr)^{(r)}(t_*)\sqrt{w(t_*)}
\ \ge\ \bar{c}\; m^{\,r+\alpha/2}\, N^{-\alpha/2}
\ \ge\ \bar{c}\, N^{\,r} ,
\qquad
\bar{c}=\frac{0.7039}{\Gamma(\alpha+r+1)}\, 12^{-\alpha/2} ;
$$
in particular, all the terms of the block at this point have one and the same sign, and for every
$\Omega\subset[N+r,3N+r]$ with $\card(\Omega)=N$
$$
\bigl\|f^{(r)} \sqrt{w}\bigr\|_C \ \ge\ \bar{c}\, N^{\,r+1} .
$$
\end{lemma}

\textit{Proof.} The first inequality is an estimate of a \emph{single} term at the point
$t_*=1/(12N)$ --- a point of guaranteed sign, not a point of maximum. On the block
$\Omega\subset[N+r,3N+r]$ the index $m=k-r$ runs over $[N,3N]$, so that $m\,t_*\le1/4$. For
$\beta=\alpha+r\ge0$ and $1\le j\le m$
$$
\binom{m+\beta}{m-j}\Big/\binom{m+\beta}{m}
=\frac{m!}{(m-j)!}\cdot\frac{1}{(\beta+1)_j}\ \le\ \frac{m^{\,j}}{j!} ,
$$
since $m!/(m-j)!\le m^{j}$ and $(\beta+1)_j\ge j!$; hence
$L^{(\beta)}_m(t)\ge\binom{m+\beta}{m}\bigl(2-I_0(2\sqrt{mt})\bigr)$, and for $m\,t\le1/4$ the bracket
is no less than $2-I_0(1)>0$, that is, the discarded alternating tail is majorized by the quantity
$I_0(1)-1$, where $I_0(1)=1.2660658\ldots$ (see\ \cite[p.\ 416]{AbrSt}).
Substitution into (\ref{reduction}) together with $k!/(k-r)!\ge m^{r}$ and
$\Gamma(m+\beta+1)/m!\ge m^{\beta}$ gives the factor $m^{\,r+\alpha/2}$, while
$t_*^{\alpha/2}=(12N)^{-\alpha/2}$ and $e^{-t_*/2}\ge e^{-1/24}$ give the constant $\bar c$,
since $\bigl(2-I_0(1)\bigr)e^{-1/24}=0.70398\ldots>0.7039$. The second inequality
is obtained, in view of $m\ge N$, by summing $N$ terms of one and the same constant sign.
\\
$\Box$

\vskip 2mm

The lower estimate is obtained from lemma \ref{lemma_BE} on a single block, whose length and position
are determined by the noise level and are not chosen freely.

\begin{theorem}\label{Th_low}
Let $\alpha\ge0$ be real, $1\le s\le\infty$, $1\le p\le\infty$ and
$\mu>r+1-1/s$. Then for every $0<\delta<1$
$$
\mathcal{E}^r_\delta\bigl(W^{\mu}_{s}, \Psi, C, \ell_p\bigr)
\ \ge\ c_1\, \delta^{\,(\mu+1/s-1-r)/(\mu+1/s-1/p)} ,
$$
where
$c_1 = \bar{c}\; 2^{-(\mu+1/s-1-r)}\, 3^{-\mu}\, r^{-(\mu+1/s-1-r)}$,
and $\bar{c}$ is the constant of lemma \ref{lemma_BE}.
\end{theorem}

\textit{Proof.} First of all we note that for $\mu>r+1-1/s$ and $1/p\le1$
all the powers below are defined.

\emph{Block level.} Let first $0<\delta\le\delta_*:=3^{-\mu}r^{\,-(\mu+1/s-1/p)}$. Put
$x:=\bigl(3^{-\mu}\delta^{-1}\bigr)^{1/(\mu+1/s-1/p)}$ and $N_1:=\lceil x\rceil$. The condition
$\delta\le\delta_*$ is exactly $x\ge r$, so that $N_1\ge r\ge1$. Hence the restriction
$1\le r\le N_1$ of lemma \ref{lemma_BE} is satisfied --- and, since $x\ge1$, also
$N_1\le x+1\le 2x$. The quantity $N_1$ is not free here: it is
determined by the level $\delta$.

\emph{Function.} Put
$$
a:=3^{-\mu}N_1^{-\mu-1/s} , \qquad
\Omega_1:=\{N_1+r,\ N_1+r+1,\ \ldots,\ 2N_1+r-1\} , \qquad
f_1:=a\sum_{k\in\Omega_1}\ell^{(\alpha)}_k .
$$
Then $\card(\Omega_1)=N_1$ and $\Omega_1\subset[N_1+r,3N_1+r]$, that is, lemma \ref{lemma_BE}
is applicable to $\Omega_1$ with $N=N_1$.

\emph{Membership in the class.} All the indices of the block satisfy
$1\le N_1+r\le k\le 2N_1+r-1\le 3N_1-1<3N_1$. For $1\le s<\infty$ we obtain
$$
\|f_1\|^s_{s,\mu}=a^s\sum_{k\in\Omega_1}k^{s\mu}
\le a^s\,N_1\,(3N_1)^{s\mu}
=3^{-s\mu}N_1^{-s\mu-1}\cdot N_1\cdot 3^{s\mu}N_1^{s\mu}=1 ,
$$
while for $s=\infty$, when $a=3^{-\mu}N_1^{-\mu}$, we have
$$
\|f_1\|_{\infty,\mu}=a\max_{k\in\Omega_1}k^{\mu}\le a\,(3N_1)^{\mu}=1 .
$$
Hence $\pm f_1\in W^\mu_s$.

\emph{Data.} The sequence $\mathbf{f}_1$ has exactly $N_1$ nonzero coordinates, all equal to
$a$, so that for all $1\le p\le\infty$
$$
\|\mathbf{f}_1\|_{\ell_p}=a\,N_1^{1/p}=3^{-\mu}N_1^{-(\mu+1/s-1/p)}
\le 3^{-\mu}x^{-(\mu+1/s-1/p)}=\delta ,
$$
where $N_1\ge x$. This is precisely the condition by which $N_1$ is tied to $\delta$.

\emph{Indistinguishability.} Let $\psi G\in\Psi$ be arbitrary and $g:=\psi(\mathbf{0})$. The pairs
$(f,\bxi)=(f_1,-\mathbf{f}_1)$ and $(f,\bxi)=(-f_1,\mathbf{f}_1)$ are both admissible --- the functions lie
in the class, and $\|\mp\mathbf{f}_1\|_{\ell_p}\le\delta$ --- and both give one and the same data
$Gf+\bxi=\mathbf{0}$. Consequently,
$$
2\,e_{\delta}\bigl(W_{s}^\mu, \psi G, C, \ell_p\bigr)
\ \ge\ \bigl\|(f^{(r)}_1-g)\sqrt{w}\bigr\|_C+\bigl\|(-f^{(r)}_1-g)\sqrt{w}\bigr\|_C
\ \ge\ 2\bigl\|f^{(r)}_1\sqrt{w}\bigr\|_C ,
$$
and, since $\psi$ was arbitrary --- neither linearity nor continuity is
required of it, --- it follows that $\mathcal{E}^r_\delta\ge\|f^{(r)}_1\sqrt{w}\|_C$.

\emph{Estimate of $\bigl\|f^{(r)}_1\sqrt{w}\bigr\|_C$ from below.} By lemma \ref{lemma_BE},
applied to $\Omega_1$ with $N=N_1$,
$$
\bigl\|f^{(r)}_1\sqrt{w}\bigr\|_C\ \ge\ a\,\bar{c}\,N_1^{\,r+1}
=3^{-\mu}\bar{c}\,N_1^{-(\mu+1/s-1-r)} .
$$
Finally, $\mu+1/s-1-r>0$ and $N_1\le 2x$ give
$$
\begin{array}{ll}
N_1^{-(\mu+1/s-1-r)} &\ \ge\ (2x)^{-(\mu+1/s-1-r)} \\
&\ =\ 2^{-(\mu+1/s-1-r)}\;3^{\mu\frac{\mu+1/s-1-r}{\mu+1/s-1/p}}\;
\delta^{\frac{\mu+1/s-1-r}{\mu+1/s-1/p}} ,
\end{array}
$$
which together with the two preceding inequalities gives, for $0<\delta\le\delta_*$,
$$
\begin{array}{ll}
\mathcal{E}^r_\delta &
\ \ge\ \bar{c}\;2^{-(\mu+1/s-1-r)}\;
3^{-\mu\left(1-\frac{\mu+1/s-1-r}{\mu+1/s-1/p}\right)}\;
\delta^{\frac{\mu+1/s-1-r}{\mu+1/s-1/p}} \\
&\ \ge\ c_1\,\delta^{\frac{\mu+1/s-1-r}{\mu+1/s-1/p}} ,
\end{array}
$$
since
$r^{-(\mu+1/s-1-r)}\le1\le3^{\mu(\mu+1/s-1-r)/(\mu+1/s-1/p)}$.

\emph{Remaining $\delta$.} For $\delta_*\le\delta<1$ the quantity $\mathcal{E}^r_\delta$ does not decrease in $\delta$
(the set of admissible perturbations grows), while the substitution $\delta=\delta_*$ into the first line
of the preceding estimate gives exactly $c_1$; therefore
$\mathcal{E}^r_\delta\ge\mathcal{E}^r_{\delta_*}\ge c_1\ge c_1\,\delta^{(\mu+1/s-1-r)/(\mu+1/s-1/p)}$.

Together the two cases cover the whole interval
$0<\delta<1$. \\
$\Box$

\vskip 2mm

\begin{corollary}\label{Cor_opt}
Let $\alpha\ge0$, $1\le s\le\infty$, $1\le p\le\infty$, $\mu>r+1-1/s$ and
$\theta>\mu+1/s-1-r$ be fixed. Then the methods (\ref{ModVer}) from
$\mathcal{S}^\theta$, at every integer truncation level $N\ge r$ with
$N\asymp\delta^{-1/(\mu+1/s-1/p)}$, are order-optimal in
accuracy: for every $0<\delta<1$
$$
\mathcal{E}^r_\delta\bigl(W^{\mu}_{s}, \Psi, C, \ell_p\bigr)\ \asymp\ \delta^{\,(\mu+1/s-1-r)/(\mu+1/s-1/p)} ,
$$
moreover the upper bound is attained on these methods.
\end{corollary}

\textit{Proof.} The lower bound is given by theorem \ref{Th_low}. For the upper one we note that
an element of the class $\Psi$ is not the operator (\ref{ModVer}) itself, which gives an approximation to $f$, but the
mapping
$$
\psi^N:\ \mathbf{y}\ \longmapsto\ \bigl(S^{\nu}_N\mathbf{y}\bigr)^{(r)}
=\sum_{k=r}^{N}\nu^N_k\, y_k\,\bigl(\ell^{(\alpha)}_k\bigr)^{(r)} ,
$$
defined on the whole of $\ell_{\max\{p,2\}}$ and giving an approximation to $f^{(r)}$: this is a finite
linear combination, so that $\psi^N(\mathbf{y})\in C$. The noise level enters not the mapping
itself, but the choice of the truncation level.

Let $0<\delta<1$ and let $N\ge r$ be an arbitrary integer level admissible in the sense of
theorem \ref{Th_up}, that is,
$$
a_1\,\delta^{-1/(\mu+1/s-1/p)}\ \le\ N\ \le\ a_2\,\delta^{-1/(\mu+1/s-1/p)}
$$
with constants $a_1$, $a_2$ not depending on $\delta$. The constant $c$ of theorem
\ref{Th_up} depends on $a_1$ and $a_2$, but not on $f$, not on $\bxi$ and not on $N$, so that
the estimate of that theorem is an estimate of the quantity
$e_\delta(W^\mu_s,\psi^N G,C,\ell_p)$, and then also
$$
\mathcal{E}^r_\delta\bigl(W^{\mu}_{s}, \Psi, C, \ell_p\bigr)\ \le\
e_\delta\bigl(W^\mu_s,\psi^N G,C,\ell_p\bigr)\ \le\ c\,\delta^{(\mu+1/s-1-r)/(\mu+1/s-1/p)} .
$$

Such a level exists for every $0<\delta<1$: put $x:=\delta^{-1/(\mu+1/s-1/p)}>1$ and
$$
N(\delta):=\max\bigl\{r,\ \lceil x\rceil\bigr\} .
$$
This is an integer, $N(\delta)\ge r$, and $x\le N(\delta)\le\max\{r,2\}\,x$, since
$\lceil x\rceil<x+1<2x$ and $r<r\,x$; that is, $N(\delta)$ is admissible with $a_1=1$ and
$a_2=\max\{r,2\}$. \\
$\Box$
\vskip 2mm

Let us pass to the minimal radius of Galerkin information and to the information complexity.
Two simple properties of the quantity $R^r_{N,\delta}$ will be needed. It does not increase in $N$:
$\Psi_N\subset\Psi_{N+1}$, and an infimum over a larger set is no larger. It does not decrease in
$\delta$: as $\delta$ grows, the set of admissible perturbations in the definition of $e_\delta$ grows.
Moreover, for $\mu>r+1-1/s$ the quantity
\begin{equation}\label{Bsup}
B:=\sup\bigl\{\|f^{(r)}\sqrt{w}\|_C:\ f\in W^\mu_s\bigr\}
\end{equation}
is finite: $\|f^{(r)}\sqrt{w}\|_C\le\|(f-S_r f)^{(r)}\sqrt{w}\|_C
+|\langle f,\ell^{(\alpha)}_r\rangle|\,\bigl\| \bigl(\ell_r^{(\alpha)}\bigr)^{(r)} \sqrt{w} \bigr\|_{C}\le c\,\|f\|_{s,\mu}$
by lemma \ref{lemma_BoundT1} with $N=r$. The algorithm with $\Omega=\varnothing$ and $\psi\equiv0$ lies in
$\Psi_N$ for every $N\ge1$, and its error equals $B$; therefore
\begin{equation}\label{RleB}
R^r_{N,\delta}\bigl(W^{\mu}_{s},\Psi_N,C,\ell_p\bigr)\le B
\qquad\text{for all } N\ge1,\ \delta>0 .
\end{equation}

\vskip 2mm

\begin{theorem}\label{Th_low_R}
Let $\alpha\ge0$ be real, $1\le s\le\infty$, $1\le p\le\infty$ and
$\mu>r+1-1/s$. Then for every $0<\delta<1$ and every integer $N\ge1$
\begin{equation}\label{low_est_R}
R^r_{N,\delta}\bigl(W^{\mu}_{s},\Psi_N,C,\ell_p\bigr)\ \ge\
c\,\max\Bigl\{\delta^{\,(\mu+1/s-1-r)/(\mu+1/s-1/p)},\ N^{-(\mu+1/s-1-r)}\Bigr\} .
\end{equation}
\end{theorem}

\textit{Proof.} The first estimate follows from the inclusion $\Psi_N\subset\Psi$ and theorem
\ref{Th_low}: $R^r_{N,\delta}\ge\mathcal E^r_\delta\ge c_1\delta^{(\mu+1/s-1-r)/(\mu+1/s-1/p)}$
for all $0<\delta<1$.

We prove the second one first for $N\ge r$. Let us fix an arbitrary algorithm
$\psi G_{\widehat\Omega}\in\Psi_N$, $\card(\widehat\Omega)\le N$. The segment $[N+r,3N+r]$ contains
$2N+1$ integers, among which $\widehat\Omega$ contains at most $N$; let us choose a set
$\Lambda_N\subset[N+r,3N+r]\setminus\widehat\Omega$ of exactly $N$ numbers and put
$$
f_2(t):=4^{-\mu}N^{-\mu-1/s}\sum_{k\in\Lambda_N}\ell^{(\alpha)}_k(t) .
$$
Since $k\le3N+r\le4N$ for $k\in\Lambda_N$, we have
$\|f_2\|^s_{s,\mu}\le4^{-s\mu}N^{-s\mu-1}\cdot N\cdot(4N)^{s\mu}=1$ for $s<\infty$ and
$\|f_2\|_{\infty,\mu}\le4^{-\mu}N^{-\mu}(4N)^{\mu}=1$, that is, $\pm f_2\in W^\mu_s$. By lemma
\ref{lemma_BE}
$$
\bigl\|f_2^{(r)}\sqrt{w}\bigr\|_C\ \ge\ 4^{-\mu}\bar c\,N^{\,r+1-\mu-1/s}
=4^{-\mu}\bar c\,N^{-(\mu+1/s-1-r)} .
$$
The Fourier--Laguerre coefficients of the functions $\pm f_2$ are nonzero only for $k\in\Lambda_N$, while
$\Lambda_N\cap\widehat\Omega=\varnothing$; therefore, for exact data ($\bxi=\mathbf 0$, which
is admissible for every $\delta>0$) the algorithm receives on $f_2$ and on $-f_2$ one and the same data
$G_{\widehat\Omega}(\pm f_2)=\mathbf 0$ and returns one and the same element $u=\psi(\mathbf 0)$.
If $\|u\sqrt{w}\|_C=\infty$, then $e_\delta=\infty$ and there is nothing to prove; otherwise, by the triangle
inequality
$$
2\bigl\|f_2^{(r)}\sqrt{w}\bigr\|_C\le\bigl\|(f_2^{(r)}-u)\sqrt{w}\bigr\|_C
+\bigl\|(-f_2^{(r)}-u)\sqrt{w}\bigr\|_C
\le2\,e_\delta\bigl(W^{\mu}_{s},\psi G_{\widehat\Omega},C,\ell_p\bigr) .
$$
The right-hand side of the obtained estimate $e_\delta\ge4^{-\mu}\bar c\,N^{-(\mu+1/s-1-r)}$ depends neither
on $\psi$ nor on $\widehat\Omega$; passing to the infimum over $\Psi_N$, we obtain the second estimate for
$N\ge r$. For $1\le N<r$ (this case is possible only for $r\ge2$), by monotonicity in $N$,
$$
R^r_{N,\delta}\ge R^r_{r,\delta}\ge4^{-\mu}\bar c\,r^{-(\mu+1/s-1-r)}
\ge4^{-\mu}\bar c\,r^{-(\mu+1/s-1-r)}\,N^{-(\mu+1/s-1-r)} ,
$$
since $N^{-(\mu+1/s-1-r)}\le1$. This is what had to be proved.\\
$\Box$

\vskip 2mm

\begin{theorem}\label{Th_R_order}
Under the hypotheses of theorem \ref{Th_low_R}, for every $0<\delta<1$ and every integer $N\ge1$,
such that $N\asymp\delta^{-1/(\mu+1/s-1/p)}$, there holds
$$
R^r_{N,\delta}\bigl(W^{\mu}_{s},\Psi_N,C,\ell_p\bigr)\ \asymp\
\delta^{\,(\mu+1/s-1-r)/(\mu+1/s-1/p)}\ \asymp\ N^{-(\mu+1/s-1-r)} ,
$$
moreover the upper estimate is attained on the methods (\ref{ModVer}) from $\mathcal S^\theta$,
$\theta>\mu+1/s-1-r$.
\end{theorem}

\textit{Proof.} The lower estimate is theorem \ref{Th_low_R}. Let
$a_1\delta^{-1/(\mu+1/s-1/p)}\le N\le a_2\delta^{-1/(\mu+1/s-1/p)}$. If $N\ge r$, then the
method $S^\nu_N$, understood as the mapping $\psi^N$ from the proof of corollary
\ref{Cor_opt}, reads only the coefficients with indices $r,\dots,N$, since
$(\ell^{(\alpha)}_k)^{(r)}\equiv0$ for $k<r$; hence $\card(\Omega)=N-r+1\le N$ and
$\psi^N G_\Omega\in\Psi_N$, and by theorem \ref{Th_up} its error does not exceed
$c\,\delta^{(\mu+1/s-1-r)/(\mu+1/s-1/p)}$ with a constant depending on $a_1$ and $a_2$, whence
the upper estimate follows. If, however, $N<r$, then $\delta^{-1/(\mu+1/s-1/p)}<r/a_1$, that is, $\delta$
is bounded away from zero, and by (\ref{RleB})
$R^r_{N,\delta}\le B\le B\,(r/a_1)^{\mu+1/s-1-r}\,\delta^{(\mu+1/s-1-r)/(\mu+1/s-1/p)}$.
The relation $\delta^{(\mu+1/s-1-r)/(\mu+1/s-1/p)}\asymp N^{-(\mu+1/s-1-r)}$ is the rewritten
condition on $N$.\\
$\Box$

\vskip 2mm

\begin{corollary}\label{Cor_Nrange}
Under the hypotheses of theorem \ref{Th_low_R} let $0<\delta<1$.

{\rm (i)} For every $A>0$ there exists $\varepsilon>0$, depending only on $A$, $\alpha$, $r$,
$\mu$, $s$, $p$, such that for all integers $1\le N\le\varepsilon\,\delta^{-1/(\mu+1/s-1/p)}$
$$
R^r_{N,\delta}\bigl(W^{\mu}_{s},\Psi_N,C,\ell_p\bigr)\ \ge\
A\,\delta^{\,(\mu+1/s-1-r)/(\mu+1/s-1/p)} ;
$$
in particular, if integers $N=N(\delta)\ge1$ are such that
$N(\delta)=o\bigl(\delta^{-1/(\mu+1/s-1/p)}\bigr)$ as $\delta\to0$, then
$$
R^r_{N(\delta),\delta}\bigl(W^{\mu}_{s},\Psi_{N(\delta)},C,\ell_p\bigr)\Big/
\delta^{\,(\mu+1/s-1-r)/(\mu+1/s-1/p)}\ \longrightarrow\ \infty ,
\qquad \delta\to0 ,
$$
that is, on such a rule for the choice of the truncation level the optimal order of accuracy
is unattainable.

{\rm (ii)} For every $c_0>0$ there exists $C_0>0$, depending only on $c_0$, $\alpha$, $r$,
$\mu$, $s$, $p$, such that for all integers $N\ge c_0\,\delta^{-1/(\mu+1/s-1/p)}$
$$
R^r_{N,\delta}\bigl(W^{\mu}_{s},\Psi_N,C,\ell_p\bigr)\ \le\
C_0\,\delta^{\,(\mu+1/s-1-r)/(\mu+1/s-1/p)} .
$$
\end{corollary}

\textit{Proof.} {\rm (i)} Let $c$ be the constant of the estimate (\ref{low_est_R}) and
$\varepsilon:=(c/A)^{1/(\mu+1/s-1-r)}$. For $N\le\varepsilon\,\delta^{-1/(\mu+1/s-1/p)}$
according to (\ref{low_est_R}) we have
$$
R^r_{N,\delta}\ \ge\ c\,N^{-(\mu+1/s-1-r)}\ \ge\
c\,\bigl(\varepsilon\,\delta^{-1/(\mu+1/s-1/p)}\bigr)^{-(\mu+1/s-1-r)}
=A\,\delta^{(\mu+1/s-1-r)/(\mu+1/s-1/p)} .
$$
For the last assertion fix $A>0$ and take the $\varepsilon>0$ just found; it does not depend
on $\delta$. Since $N(\delta)\,\delta^{1/(\mu+1/s-1/p)}\to0$, for all sufficiently small
$\delta$ we have $N(\delta)\le\varepsilon\,\delta^{-1/(\mu+1/s-1/p)}$ and therefore
$R^r_{N(\delta),\delta}\ge A\,\delta^{(\mu+1/s-1-r)/(\mu+1/s-1/p)}$. As $A>0$ is arbitrary,
the ratio tends to infinity.

{\rm (ii)} Put $n:=\min\{N,\lceil\delta^{-1/(\mu+1/s-1/p)}\rceil\}$; then $n\le N$ and
$\min\{c_0,1\}\,\delta^{-1/(\mu+1/s-1/p)}\le n\le2\,\delta^{-1/(\mu+1/s-1/p)}$, where the last
inequality follows from $\delta^{-1/(\mu+1/s-1/p)}>1$. If $n\ge r$, then by monotonicity and
theorem \ref{Th_R_order}, applied to $n$ with the constants $a_1=\min\{c_0,1\}$ and $a_2=2$ and with
the Fourier method $\nu^N_k\equiv1$ (it lies in $\mathcal S^\theta$ for $\theta=\mu+1/s-r$ with $C(\nu)=1$,
so that the constant depends only on $\alpha$, $r$, $\mu$, $s$, $p$, $a_1$, $a_2$),
$R^r_{N,\delta}\le R^r_{n,\delta}\le c\,\delta^{(\mu+1/s-1-r)/(\mu+1/s-1/p)}$. If $n<r$, then
$\min\{c_0,1\}\,\delta^{-1/(\mu+1/s-1/p)}<r$, whence
$\delta>\delta_0:=(\min\{c_0,1\}/r)^{\mu+1/s-1/p}$, and by (\ref{RleB})
$R^r_{N,\delta}\le B\le B\,\delta_0^{-(\mu+1/s-1-r)/(\mu+1/s-1/p)}\,
\delta^{(\mu+1/s-1-r)/(\mu+1/s-1/p)}$. It remains to take $C_0$ equal to the larger of the two
constants.\\
$\Box$

\vskip 2mm

\begin{remark}\label{rem_onemeth}\rm
As follows from theorem \ref{Th_R_order}, a single method (\ref{ModVer}) with a suitable $\theta$
makes it possible to ensure optimal accuracy of recovery of all admissible derivatives, since
the truncation order $N$ does not depend on $r$.
\end{remark}
\vskip 2mm

\begin{theorem}\label{Th_Nmin}
Under the hypotheses of theorem \ref{Th_low_R} let the constant $K$ in the definition (\ref{Nmin}) be no smaller
than the constant $C_0$ of corollary \ref{Cor_Nrange}{\rm (ii)} with $c_0=1$. Then the quantity
$N_{\min}(\delta)$ is defined for every $0<\delta<1$, and
\begin{equation}\label{optNp}
N_{\min}(\delta)\ \asymp\ \delta^{-1/(\mu+1/s-1/p)} ,
\qquad
R^r_{N_{\min},\delta}\bigl(W^{\mu}_{s},\Psi_{N_{\min}},C,\ell_p\bigr)\ \asymp\
\delta^{\,(\mu+1/s-1-r)/(\mu+1/s-1/p)} .
\end{equation}
\end{theorem}

\textit{Proof.} By corollary \ref{Cor_Nrange}{\rm (ii)} with $c_0=1$ the number
$N=\lceil\delta^{-1/(\mu+1/s-1/p)}\rceil$ satisfies
$R^r_{N,\delta}\le C_0\,\delta^{(\mu+1/s-1-r)/(\mu+1/s-1/p)}\le K\,\delta^{(\mu+1/s-1-r)/(\mu+1/s-1/p)}$,
so that the set in (\ref{Nmin}) is nonempty and
$N_{\min}\le\lceil\delta^{-1/(\mu+1/s-1/p)}\rceil\le2\,\delta^{-1/(\mu+1/s-1/p)}$. In view of
monotonicity in $N$, this set contains, together with each of its elements, all larger
indices as well, so that $N_{\min}$ is a threshold. By corollary \ref{Cor_Nrange}{\rm (i)} with $A=2K$, for all
$N\le\varepsilon\,\delta^{-1/(\mu+1/s-1/p)}$ we have
$R^r_{N,\delta}\ge2K\,\delta^{(\mu+1/s-1-r)/(\mu+1/s-1/p)}>K\,\delta^{(\mu+1/s-1-r)/(\mu+1/s-1/p)}$,
that is, such $N$ do not belong to this set, and
$N_{\min}>\varepsilon\,\delta^{-1/(\mu+1/s-1/p)}$. The second relation (\ref{optNp}):
$R^r_{N_{\min},\delta}\le K\,\delta^{(\mu+1/s-1-r)/(\mu+1/s-1/p)}$ by the definition of $N_{\min}$, while
$R^r_{N_{\min},\delta}\ge c\,\delta^{(\mu+1/s-1-r)/(\mu+1/s-1/p)}$ by theorem
\ref{Th_low_R}.\\
$\Box$

\vskip 2mm

\begin{remark}\label{rem_Nmin}\rm
In terms of the theory of information complexity, the quantity $N_{\min}$ is the information complexity
of the problem of numerical differentiation based on Galerkin information. Theorem \ref{Th_Nmin}
shows that the smallest, in order, number of perturbed Fourier--Laguerre coefficients
sufficient for the optimal order of accuracy equals $\delta^{-1/(\mu+1/s-1/p)}$, while
corollary \ref{Cor_Nrange}{\rm (i)} --- that a smaller number is not sufficient. The rule for choosing
the truncation level $N\asymp\delta^{-1/(\mu+1/s-1/p)}$ of theorems \ref{Th_up} and \ref{Th_R_order}
coincides in order with $N_{\min}$: the methods (\ref{ModVer}) use exactly (in order)
the information-complexity amount of data, and this amount does not depend on the order of the derivative $r$.
\end{remark}

\section{Computational experiments}\label{exp}

We now give numerical illustrations of the results obtained. Two examples are considered: a test
function with oscillating derivatives and the applied problem of Laplace transform inversion by
the Post--Widder formula, where derivatives of different orders of one and the same function are required.
The computations were performed on a computer with an Apple M4 processor (10 cores) and 24 GB of memory,
running macOS 26.6.1, Python 3.14.6, NumPy 2.5.1; the figures were produced with
Matplotlib 3.11.0. All
tables and both figures of this section are generated by a single script.

Throughout this section $p=s=2$, the weight is $w(t)=t^{\alpha}e^{-t}$ with $\alpha=1/2$ and $\alpha=1$, and
the input error levels in both examples are
$\delta=10^{-5},10^{-6},10^{-7},10^{-8}$. In example~1 we approximate the derivatives of orders
$r=1$ and $r=2$, and in example~2 those of orders $r=2$ and $r=3$. The experiment
follows the setting of sections \ref{prelim}--\ref{EE} literally; it is based on four conventions.

{\it Normalization.} All estimates in sections \ref{EE} and \ref{sharp} are stated on the unit ball
of the space $W^\mu_s$, and therefore each function is normalized before noise is added,
$\widehat f:=f/\|f\|_{2,\mu}$, and every number reported below refers to $\widehat f$. In
particular, $\mathbf f$ is the vector of exact Fourier--Laguerre coefficients of the function $\widehat f$,
truncated at $k=N$.

{\it Perturbed data.} $\mathbf f^\delta=\mathbf f+\bxi$, where the error
sequence is a single random realization on the sphere of radius $\delta$,
$$
\bxi=\delta\,\frac{\mathbf g}{\|\mathbf g\|_{\ell_2}},\qquad
\mathbf g=(g_0,\ldots,g_N),
$$
with independent standard normal $g_k$; then $\|\bxi\|_{\ell_2}=\delta$ exactly, that
is, condition (\ref{perturbation2}) for $p=2$ holds with equality. The realization was generated
by NumPy's PCG64 generator with seed $20260909$.

{\it Truncation level.} For $p=s=2$ theorem \ref{Th_up} prescribes
$N\asymp\delta^{-1/\mu}$, and we take $N=\lceil\delta^{-1/\mu}\rceil$, rounded up to the
nearest even integer: evenness is required by the de la Vall\'{e}e Poussin method, for which $N=2n$, and is
harmless for the Fourier sums. This rule \emph{does not depend on $r$}: one and the same $N$ serves both orders in
each example, which is exactly what remark \ref{rem_level} asserts.

{\it Error.} The reported error is the weighted uniform norm
$$
\max_{t>0}\bigl|\bigl(\widehat f-S^\nu_N\mathbf f^\delta\bigr)^{(r)}(t)\sqrt{w(t)}\bigr| ,
$$
that is, exactly the quantity estimated in theorem \ref{Th_up}. The maximum is taken over a grid
uniform in $\sqrt t$: near zero $\varphi^{(\beta)}_m$ behaves like $J_\beta(2\sqrt{mt})$,
and the wavelength in the variable $t$ decreases like $\sqrt{t/m}$, so that a grid uniform in $t$
would resolve the peak incorrectly. A fourfold refinement of the grid moves the reported maxima by no more than
$0.09\,\%$; the contribution of $t$ beyond the right endpoint of the grid does not exceed $3\cdot10^{-16}$.

\subsection{Example 1}

We consider the function
$$
f_1(t)=t^{5/2}e^{-t/2}\cos 2t .
$$
Its Fourier--Laguerre coefficients are written in closed form by means of the generating function
of the Laguerre polynomials: for $\beta=\alpha+5/2$ and $s=3/2-2i$
$$
\sum_{k\ge0}\Bigl(\int_0^\infty t^{\beta}e^{-st}L^{(\alpha)}_k(t)\,dt\Bigr)z^k
=\Gamma(\beta+1)\,s^{-(\beta+1)}(1-z)^{5/2}(1-qz)^{-(\beta+1)},\qquad q=\frac{s-1}{s} ,
$$
whence $\langle f_1,\ell^{(\alpha)}_k\rangle$ is obtained as the real part of the coefficient
of $z^k$, multiplied by $\sqrt{k!/\Gamma(k+\alpha+1)}$. Hence
$\langle f_1,\ell^{(\alpha)}_k\rangle\asymp k^{-(7/2+\alpha/2)}$, so that $\|f_1\|_{2,\mu}$
is finite if and only if $\mu<\mu_{\max}:=3+\alpha/2$: this is $3.25$ for
$\alpha=1/2$ and $3.5$ for $\alpha=1$. As the working value we take the largest number with one decimal
place strictly below this bound, that is, $\mu=3.2$ for $\alpha=1/2$ and $\mu=3.4$ for
$\alpha=1$; the condition $\mu>r+1-1/s$ holds both for $r=1$ and for $r=2$, and a single $\mu$ serves both
derivatives.

Let us note a feature of this function. The factor $e^{-t/2}\cos2t$ produces in the coefficients
a geometric decay with ratio $|q|=0.8246$, which up to $k\approx150$ masks
the power-law rate; therefore for small $N$ the approximation loses not so much from noise as from
truncation. Moreover, the normalized amplitude here is small: the peak of the quantity
$|\widehat f_1^{\,(r)}\sqrt w|$ for $\alpha=1/2$ equals $8.6\cdot10^{-5}$ for $r=1$ and
$1.9\cdot10^{-4}$ for $r=2$. Hence also the coarsest of the levels considered: at
$\delta=10^{-4}$ the quantity being recovered would be on a par with the noise.

Tables~\ref{tbl1} and \ref{tbl2} present the results of approximating the derivatives of $f_1$ by the
Fourier and de la Vall\'{e}e Poussin methods for $\alpha=1/2$ and $\alpha=1$.

\begin{table}[ht]
\caption{Approximation of $f_1'$, $f_1(t)=t^{5/2}e^{-t/2}\cos2t$. Left: $\alpha=1/2$,
$\mu=3.2$, $\|f_1\|_{2,\mu}=2.44\cdot10^{4}$; right: $\alpha=1$, $\mu=3.4$,
$\|f_1\|_{2,\mu}=7.57\cdot10^{4}$.}\label{tbl1}
\centering
\begin{tabular}{ccccccc}
\hline
 & \multicolumn{3}{c}{$\alpha=1/2$} & \multicolumn{3}{c}{$\alpha=1$} \\
$\delta$ & $N$ & Fourier & de la Vall\'{e}e P. & $N$ & Fourier & de la Vall\'{e}e P. \\
\hline
$10^{-5}$ &  38 & $2.47\cdot10^{-5}$ & $2.46\cdot10^{-5}$ &  30 & $2.46\cdot10^{-5}$ & $3.35\cdot10^{-5}$ \\
$10^{-6}$ &  76 & $7.99\cdot10^{-6}$ & $5.80\cdot10^{-6}$ &  60 & $4.99\cdot10^{-6}$ & $3.54\cdot10^{-6}$ \\
$10^{-7}$ & 154 & $1.12\cdot10^{-6}$ & $1.29\cdot10^{-6}$ & 116 & $1.75\cdot10^{-6}$ & $9.78\cdot10^{-7}$ \\
$10^{-8}$ & 318 & $3.09\cdot10^{-7}$ & $3.41\cdot10^{-7}$ & 226 & $6.37\cdot10^{-7}$ & $5.31\cdot10^{-7}$ \\
\hline
\end{tabular}
\end{table}

\begin{table}[ht]
\caption{Approximation of $f_1''$ with the same parameters and the same truncation levels as in
Table~\ref{tbl1}.}\label{tbl2}
\centering
\begin{tabular}{ccccccc}
\hline
 & \multicolumn{3}{c}{$\alpha=1/2$} & \multicolumn{3}{c}{$\alpha=1$} \\
$\delta$ & $N$ & Fourier & de la Vall\'{e}e P. & $N$ & Fourier & de la Vall\'{e}e P. \\
\hline
$10^{-5}$ &  38 & $3.39\cdot10^{-4}$ & $1.31\cdot10^{-4}$ &  30 & $1.92\cdot10^{-4}$ & $1.76\cdot10^{-4}$ \\
$10^{-6}$ &  76 & $9.24\cdot10^{-5}$ & $5.68\cdot10^{-5}$ &  60 & $5.78\cdot10^{-5}$ & $4.93\cdot10^{-5}$ \\
$10^{-7}$ & 154 & $2.24\cdot10^{-5}$ & $1.73\cdot10^{-5}$ & 116 & $5.43\cdot10^{-5}$ & $1.79\cdot10^{-5}$ \\
$10^{-8}$ & 318 & $4.79\cdot10^{-5}$ & $1.29\cdot10^{-5}$ & 226 & $3.05\cdot10^{-5}$ & $2.45\cdot10^{-5}$ \\
\hline
\end{tabular}
\end{table}

A comparison of the two tables shows exactly what they are presented for: with one and the same number
of retained coefficients the second derivative is recovered worse than the first --- by an order of magnitude
at $\delta=10^{-5}$ and by more than two orders of magnitude at $\delta=10^{-8}$ --- and its
error decreases noticeably more slowly. This agrees with the exponent
$(\mu+1/s-1-r)/(\mu+1/s-1/p)$ of theorem \ref{Th_up}, which for $\alpha=1/2$ equals $0.53$ for
$r=1$ and only $0.22$ for $r=2$.

Figure~\ref{Fig1} shows the exact derivatives $\widehat f_1^{\,(r)}\sqrt w$ and their two
approximations.

\begin{figure}[ht]
\centering
\includegraphics[width=0.98\textwidth]{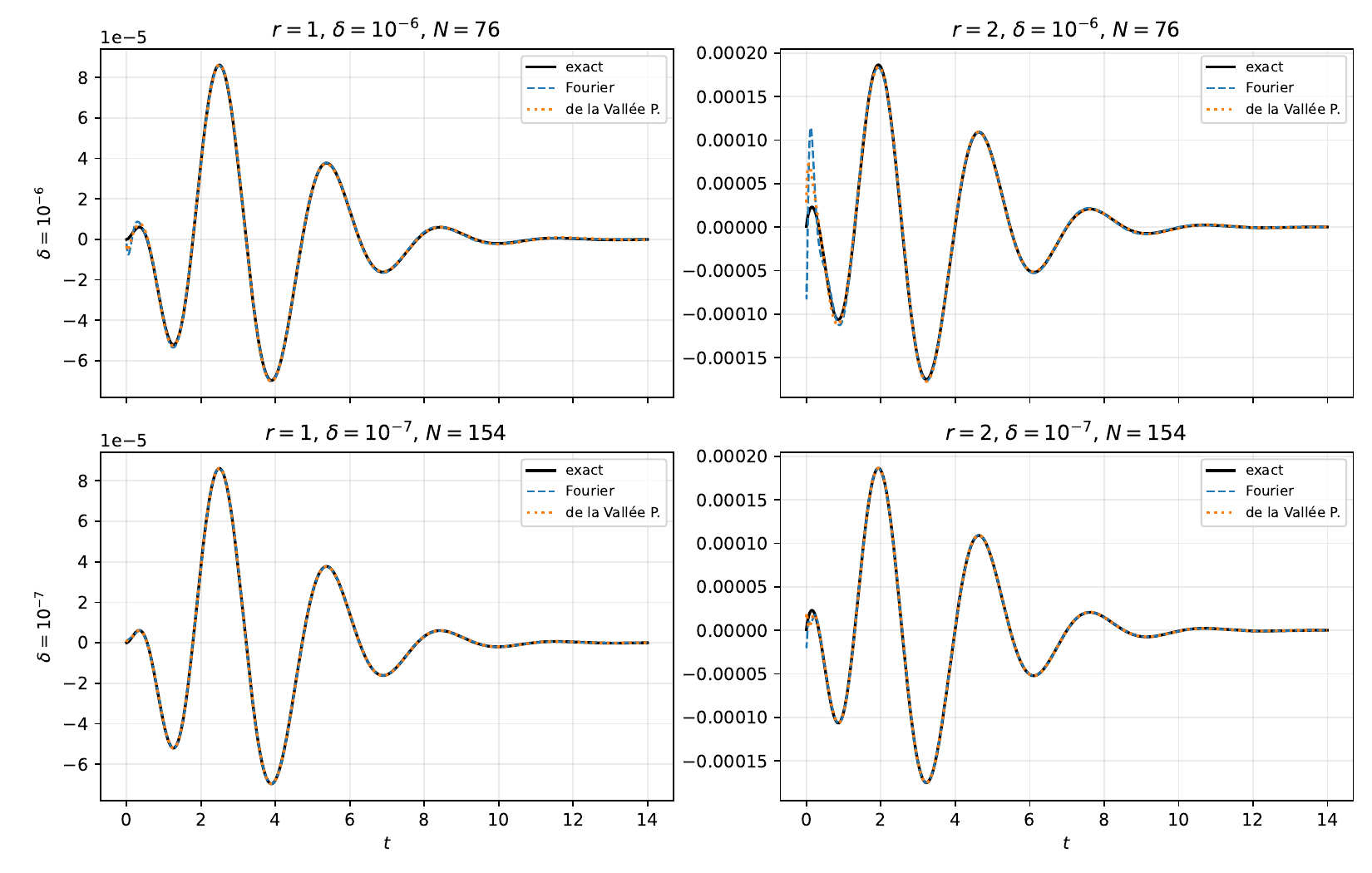}
\caption{Example 1, $\alpha=1/2$, $\mu=3.2$: the exact
$\widehat f_1^{\,\prime}\sqrt w$ (left) and $\widehat f_1^{\,\prime\prime}\sqrt w$ (right) and
their approximations by the two methods. Top row: $\delta=10^{-6}$, $N=76$; bottom:
$\delta=10^{-7}$, $N=154$. The noise realizations correspond to the second and third rows of
Table~\ref{tbl1}.}\label{Fig1}
\end{figure}

\subsection{Example 2: Laplace transform inversion}

Let $F$ be the Laplace transform of a function $f$. The Post--Widder inversion formula
\begin{equation}\label{postwidder}
P_{r}F(x)=\frac{(-1)^{r}}{r!}\Bigl(\frac{r}{x}\Bigr)^{r+1}F^{(r)}\Bigl(\frac{r}{x}\Bigr)
\ \xrightarrow[\ r\to\infty\ ]{}\ f(x)
\end{equation}
(see\ \cite[theorem 2.4]{Coh}) requires derivatives of \emph{different} orders of one and the same
function $F$, known approximately --- this is precisely the applied setting for which
remark \ref{rem_level} gives a saving: the truncation is chosen once.

Let us take $f(x)=e^{-3x}$, that is, $F(t)=1/(t+3)$; then $P_{r}F(x)=(1+3x/r)^{-(r+1)}$, and both
quantities are known exactly. The coefficients of $F$ are written as a one-dimensional integral with a positive
integrand: the substitution $(t+3)^{-1}=\int_0^\infty e^{-(t+3)v}dv$ and the identity
$\int_0^\infty e^{-pt}t^\alpha L^{(\alpha)}_k(t)\,dt=\frac{\Gamma(k+\alpha+1)}{k!}\frac{(p-1)^k}{p^{k+\alpha+1}}$
give
$$
\langle F,\ell^{(\alpha)}_k\rangle=\sqrt{\frac{\Gamma(k+\alpha+1)}{k!}}
\int_0^\infty e^{-3v}\,v^{k}(1+v)^{-(k+\alpha+1)}\,dv .
$$
By Laplace's method it follows from this that $\langle F,\ell^{(\alpha)}_k\rangle\asymp k^{\alpha/2}e^{-2\sqrt{3k}}$
--- a decay faster than any power, and therefore $F$ belongs to \emph{all} the classes $W^\mu_2$, and
$\mu$ is here assigned rather than computed. We take $\mu=4$ for both values of $\alpha$; the condition
$\mu>r+1-1/s$ holds both for $r=2$ and for $r=3$. The normalization norms equal
$1.76$ for $\alpha=1/2$ and $2.04$ for $\alpha=1$.

\begin{table}[H]
\caption{Approximation of $F''$, $F(t)=1/(t+3)$, $\mu=4$. The normalization norms:
$\|F\|_{2,\mu}=1.76$ for $\alpha=1/2$ and $2.04$ for $\alpha=1$.}\label{tbl3}
\centering
\begin{tabular}{cccccc}
\hline
 & & \multicolumn{2}{c}{$\alpha=1/2$} & \multicolumn{2}{c}{$\alpha=1$} \\
$\delta$ & $N$ & Fourier & de la Vall\'{e}e P. & Fourier & de la Vall\'{e}e P. \\
\hline
$10^{-5}$ &  18 & $4.37\cdot10^{-5}$ & $7.10\cdot10^{-4}$ & $3.65\cdot10^{-4}$ & $2.77\cdot10^{-4}$ \\
$10^{-6}$ &  32 & $6.33\cdot10^{-5}$ & $8.74\cdot10^{-5}$ & $8.97\cdot10^{-5}$ & $7.18\cdot10^{-6}$ \\
$10^{-7}$ &  58 & $2.16\cdot10^{-5}$ & $1.44\cdot10^{-5}$ & $6.12\cdot10^{-6}$ & $4.81\cdot10^{-6}$ \\
$10^{-8}$ & 100 & $5.53\cdot10^{-6}$ & $4.04\cdot10^{-7}$ & $2.40\cdot10^{-6}$ & $3.84\cdot10^{-6}$ \\
\hline
\end{tabular}
\end{table}

\begin{table}[H]
\caption{Approximation of $F^{(3)}$ with the same parameters and the same truncation levels as in
Table~\ref{tbl3}.}\label{tbl4}
\centering
\begin{tabular}{cccccc}
\hline
 & & \multicolumn{2}{c}{$\alpha=1/2$} & \multicolumn{2}{c}{$\alpha=1$} \\
$\delta$ & $N$ & Fourier & de la Vall\'{e}e P. & Fourier & de la Vall\'{e}e P. \\
\hline
$10^{-5}$ &  18 & $2.58\cdot10^{-4}$ & $3.03\cdot10^{-3}$ & $1.06\cdot10^{-3}$ & $1.23\cdot10^{-3}$ \\
$10^{-6}$ &  32 & $4.36\cdot10^{-4}$ & $5.43\cdot10^{-4}$ & $6.02\cdot10^{-4}$ & $6.80\cdot10^{-5}$ \\
$10^{-7}$ &  58 & $2.50\cdot10^{-4}$ & $1.14\cdot10^{-4}$ & $5.24\cdot10^{-5}$ & $5.40\cdot10^{-5}$ \\
$10^{-8}$ & 100 & $1.61\cdot10^{-4}$ & $1.59\cdot10^{-5}$ & $5.91\cdot10^{-5}$ & $5.90\cdot10^{-5}$ \\
\hline
\end{tabular}
\end{table}

\begin{figure}[!t]
\centering
\includegraphics[width=0.90\textwidth]{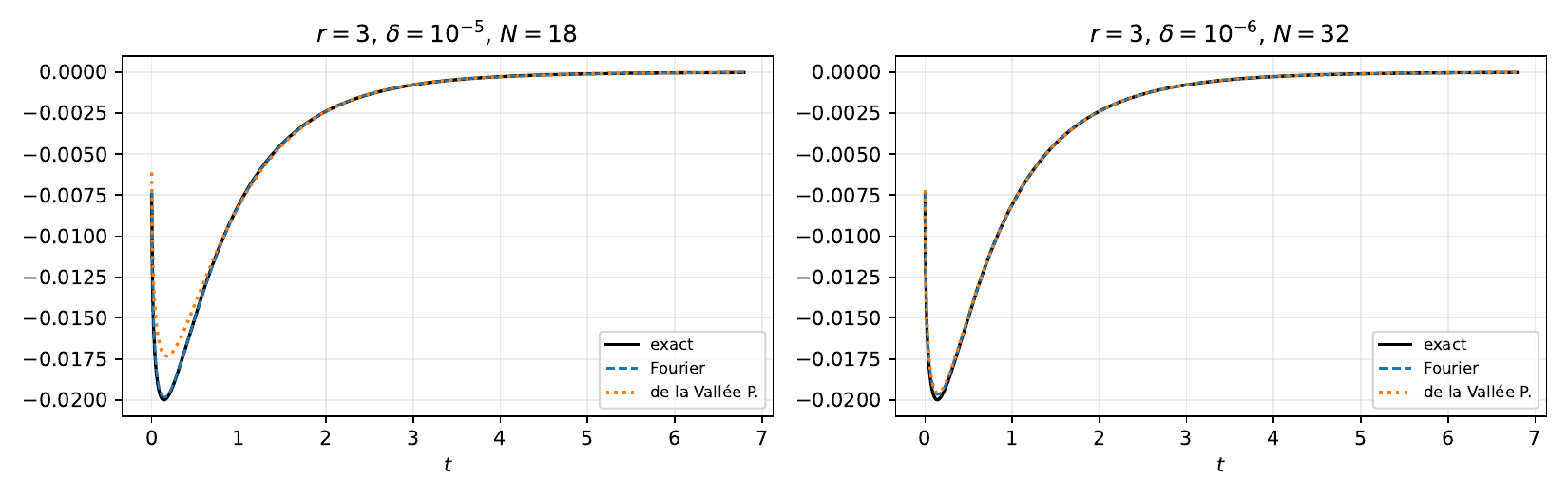}
\caption{Example 2, $\alpha=1/2$, $\mu=4$: the exact $\widehat F^{(3)}\sqrt w$ and its
approximations by the two methods. Left: $\delta=10^{-5}$, $N=18$; right:
$\delta=10^{-6}$, $N=32$. The noise realizations correspond to the first and second rows of
Table~\ref{tbl4}.}\label{Fig2}
\end{figure}

\end{document}